\documentclass{amsart}

\usepackage{hyperref}
\usepackage{pinlabel}

\newtheorem{conjecture}{Conjecture}
\newtheorem{theorem}{Theorem}
\newtheorem{que}{Question}
\newtheorem{lemma}{Lemma}
\newtheorem{claim}{Claim}

\newcommand{\Q}{\mathbb{Q}}
\newcommand{\Z}{\mathbb{Z}}
\newcommand{\floor}[1]{\left\lfloor#1\right\rfloor} 
\newcommand{\ceil}[1]{\left\lceil#1\right\rceil} 

\title{Boundary slopes (nearly) bound exceptional slopes}

\author{Kazuhiro Ichihara}
\address{Department of Mathematics, College of Humanities and Sciences, Nihon University, 3-25-40 Sakurajosui, Setagaya-ku, Tokyo 156-8550, Japan.}
\email{ichihara.kazuhiro@nihon-u.ac.jp}
\thanks{Ichihara was supported by JSPS KAKENHI Grant Number 22K03301.}

\author{Thomas W.\ Mattman}
\address{Department of Mathematics and Statistics,
California State University, Chico,
Chico, CA 95929-0525}
\email{TMattman@CSUChico.edu}

\subjclass{Primary 57K32, Secondary 57K35, 57K10}
\keywords{exceptional surgery, boundary slope}
\date{\today}

\dedicatory{Dedicated to Professor Kimihiko Motegi on his 60th birthday.}

\begin{document}

\begin{abstract}
For a hyperbolic knot in $S^3$, Dehn surgery
along slope $r \in \Q \cup \{\frac10\}$ is {\em exceptional} if
it results in a non-hyperbolic manifold. We say meridional surgery, 
$r = \frac10$, is {\em trivial} as it recovers the manifold $S^3$.
We provide evidence in support of two conjectures. 
The first (inspired by a question of Professor Motegi) states that 
there are boundary slopes $b_1 < b_2$ such that all 
non-trivial exceptional surgeries occur, as rational numbers, in the interval $[b_1,b_2]$. 
We say a boundary slope is {\em NIT} if it is non-integral or toroidal. Second, when there are non-trivial exceptional surgeries, we conjecture there are NIT boundary slopes $b_1 \leq b_2$ so that the exceptional surgeries lie in 
$[\floor{b_1},\ceil{b_2}]$. Moreover, if $\ceil{b_1} \leq \floor{b_2}$, the integers in the interval $[ \ceil{b_1}, \floor{b_2} ]$ are all exceptional surgeries.
\end{abstract}

\maketitle

\section{Introduction}

Let $K$ be a hyperbolic knot in $S^3$. Using standard coordinates, we can identify Dehn surgery slopes with elements of $\Q \cup \{ \frac10 \}$,
where $\frac10$ corresponds to meridional surgery. We refer to meridional surgery as {\em trivial} and
say that $r \in \Q$ is a non-trivial {\em exceptional} surgery slope if Dehn surgery along $r$ results in a non-hyperbolic manifold. 

In this paper, we provide evidence in support of the following conjecture, originally stated as a question by Professor Motegi (see \cite{MattmanCyclic}).

\begin{conjecture}
\label{Conj1}%
For each hyperbolic knot in $S^3$, there exists a pair of boundary slopes $b_1, b_2$ with $b_1 < b_2$ such that all exceptional surgeries occur, as rational numbers, in the interval $[b_1,b_2]$. 
\end{conjecture}

Here a \textit{boundary slope} is a slope represented by the boundary curves of an essential embedded surface in the exterior of a knot. 

As suggested in the earlier paper~\cite{MattmanCyclic}, it seems likely that more can be said. 
For the figure-eight knot, the $(-2,3,7)$-pretzel knot, the twist knots, and the  $(-3,3,n)$-pretzel knots, 
the exceptional surgeries occur as a sequence of rational numbers bounded above and below by boundary slopes. In fact, except for the half-integer
toroidal boundary slope of the $(-2,3,7)$-knot,
in these examples the non-trivial exceptional slopes are a sequence of consecutive integers bounded between two boundary slopes.

Others have made similar observations. Notably, Teragaito~\cite{Teragaito} conjectured that integral exceptional surgeries occur as a 
sequence of consecutive integers. Dunfield~\cite{DunfieldInvent} showed that, for small knots, if surgery along slope $r$ yields a manifold with cyclic fundamental group, then 
there is a non-integral boundary slope in the interval $(r-1,r+1)$. This was generalized in \cite{IMS} where it's shown that surgeries that result in 
a manifold with finite fundamental group or a Seifert fibered space are also near boundary slopes.

Recently Dunfield~\cite{DunfieldCensus} made a complete enumeration of all exceptional fillings of $1$-cusped manifolds with an ideal triangulation
of nine or fewer tetrahedra. This includes data for 1267 complements of hyperbolic knots in $S^3$.  Making use of this data we propose the following 
refinement of the earlier conjectures. We will say that a boundary slope is {\em NIT} if it is non-integral or toroidal. 
We use $\floor{b}$ (resp., $\ceil{b}$) to denote the floor (resp., ceiling) of 
$b \in \Q$. If $b \in \Z$, $\floor{b} = \ceil{b} = b$. If not,
$b$ is between the consecutive integers $\floor{b}$ and $\ceil{b}$: 
$\floor{b} < b < \ceil{b}$.

\begin{conjecture}
\label{Conj6}%
Let $K$ be a hyperbolic knot in $S^3$ that admits non-trivial exceptional surgeries.
There are (possibly equal) NIT boundary slopes $b_1 \leq b_2$ such that all exceptional surgeries occur as rational numbers in the interval $[\floor{b_1},\ceil{b_2}]$ 
(or in the set $\{ \floor{b_1} \}$ if $\floor{b_1} = \ceil{b_2}$) 
and, if $\ceil{b_1} \leq \floor{b_2}$, the integers in the interval $[ \ceil{b_1}, \floor{b_2} ]$ (or the set $\{ \ceil{b_1} \}$ in case $\ceil{b_1} = \floor{b_2}$) are all exceptional surgeries.
\end{conjecture}

Note that, since $b_1 \leq b_2$, $\floor{b_1} = \ceil{b_2}$ implies $b_1 = b_2$ is an integer.
For the reader's convenience, here is an equivalent statement of the conjecture that may be easier to parse.

\begin{conjecture}
\label{Conj5}%
Let $K$ be a hyperbolic knot in $S^3$ that admits non-trivial exceptional surgeries. One of the following occurs.
\begin{enumerate}
\item
There are (possibly equal) integral toroidal boundary slopes $b_1 \leq b_2$. All exceptional slopes are rational numbers in the interval $[b_1, b_2]$ (or the set $\{b_1\}$ if $b_1 = b_2$) and every integer in that interval (resp., set) is exceptional.
\item
There is a non-integral boundary slope $b_1$ and an integral toroidal boundary slope $b_2$.
If $b_1 < b_2$, (resp. $b_2 < b_1$), every exceptional slope is in the interval $[ \floor{b_1}, b_2 ]$ (resp. $[ b_2, \ceil{b_1} ]$) and every integer in the interval $[ \ceil{b_1}, b_2 ]$ (resp. $[ b_2, \floor{b_1} ]$) is exceptional.
\item
There are (possibly equal) non-integral boundary slopes $b_1 \leq b_2$ and the exceptional slopes are in the interval $[ \floor{b_1}, \ceil{b_2} ]$. If $\ceil{b_1} = \floor{b_2}$, then that integer is an exceptional surgery. If $\ceil{b_1} < \floor{b_2}$, then every integer in $[\ceil{b_1}, \floor{b_2}]$ is exceptional.
\end{enumerate}
\end{conjecture}

In Section~\ref{SecCensus} below, we present examples showing that every case of 
Conjecture~\ref{Conj5} arises. However, as explained further in that section, we
do have a question about part (3) of the Conjecture.

\begin{que}
\label{queConj53}%
Is there a knot in $S^3$ which satisfies Conjecture~\ref{Conj5} (3), but {\em only}
with the choice $b_1 < b_2$?
\end{que}

Not only does Dunfield's data make these conjectures plausible, 
we can prove them for alternating knots, Montesinos knots, and some torti-rational knots. 
(See the next section for the definition of torti-rational knots.)

\begin{theorem}\label{Thm7tet}
Conjectures~\ref{Conj1} and \ref{Conj6} hold for knots that admit a triangulation with at most seven tetrahedra.
\end{theorem}

\begin{theorem}\label{ThmAlt}
For a hyperbolic alternating knot with non-trivial exceptional slopes, one of the following holds. 
(i) There is a toroidal boundary slope which is the only exceptional slope. 
(ii) There are a pair of integral toroidal boundary slopes, and the exceptional slopes for $K$ are the integers contained in the interval bounded by them. 
Thus Conjectures~\ref{Conj1} and \ref{Conj6} hold for hyperbolic alternating knots.
\end{theorem}

\begin{theorem}
\label{ThmMont}
Conjectures~\ref{Conj1} and \ref{Conj6} hold for hyperbolic Montesinos knots. 
Moreover, for a hyperbolic Montesinos knot with non-trivial exceptional slopes, one of (1) or (2) stated in Conjecture \ref{Conj5} holds. 
\end{theorem}

\begin{theorem}\label{ThmTorti}
Let $K$ be a hyperbolic torti-rational knot $K(\beta/\alpha;n)$ with an irreducible fraction $\beta/\alpha$ and $n \ge 4$. 
Suppose that $K$ admits a non-trivial exceptional slope. 
Then one of the following holds. 
(i) There is a toroidal boundary slope for $K$ which is the only non-trivial exceptional slope. 
(ii) There are a pair of integral toroidal boundary slopes, and the
non-trivial exceptional slopes for $K$ are the integers contained in the interval bounded by them. 
Thus Conjectures~\ref{Conj1} and \ref{Conj6} hold for the torti-rational knots. 
\end{theorem}

We say that a Dehn surgery slope $r \in \Q$ is a {\em cyclic} (respectively {\em finite}) slope if the surgery results in a manifold having cyclic (resp.~finite) 
fundamental group. Teragaito~\cite{TeragaitoIsol} produced several infinite families of knots that have an exceptional 
integral surgery $m$ such that neither $m-1$ nor $m+1$ is 
exceptional. For such a knot, Conjecture~\ref{Conj6} would imply that $m$ is the only integral exceptional surgery and either $m$ is a toroidal slope, or else
there are (possibly equal) non-integral boundary slopes $b_1 \leq b_2$ such that $m \in [\lfloor b_1 \rfloor, \lceil b_2 \rceil ]$. In particular, we can prove Conjecture~\ref{Conj6} for knots with a single exceptional surgery, which is either toroidal or cyclic. 
In addition to the infinite families of Teragaito's paper~\cite{TeragaitoIsol}, there are many examples of knots whose
only non-trivial exceptional surgery is toroidal, including two-bridge knots
$K_{[b_1,b_2]}$ with $|b_1|,|b_2| > 2$ and
pretzel knots $P(q_1,q_2,q_3)$ with
$q_j \neq 0, \pm 1$ for $j = 1,2,3$. 
(See the proof of Theorem 2 in Section 2.)

\begin{theorem}\label{ThmIsol}
Let $K$ be a hyperbolic knot in $S^3$ that has a single non-trivial exceptional surgery, which is toroidal or cyclic. Then $K$ satisfies Conjecture~\ref{Conj6}
\end{theorem}

\begin{proof}
If $K$ has the toroidal slope $m$ as its unique non-trivial exceptional surgery, let $b_1 = b_2 = m$.
Suppose the cyclic slope $m$ is the only non-trivial exceptional surgery. Dunfield~\cite{DunfieldInvent} showed that there is a boundary slope
$b \in (m-1,m+1)$. Let $b_1 = b_2 = b$.
\end{proof}

The figure-eight knot is well-known as the hyperbolic knot in $S^3$ that admits the largest number of exceptional surgeries,
including every integer in the interval $[-4,4]$. 
Conjecture~\ref{Conj6} shows how this is related to the figure-eight knot being amphicheiral. 

\begin{theorem}
If $K$ is an amphicheiral knot for which Conjecture~\ref{Conj6} holds and $m$ is an integral exceptional surgery slope, then $|m| \leq 4$.
\end{theorem}

\begin{proof} 
Since $K$ is amphicheiral, we can assume $m \geq 0$.
The conjecture implies the $2m+1$ integers in $[-m, m]$ are all exceptional. If $m > 4$, $K$ 
would have more than ten exceptional
surgeries, contradicting the bound proved by
Lackenby and Meyerhoff~\cite{LackenbyMeyerhoff}.
\end{proof}

The two-bridge knots $K_{[2b,2b]}$ for $b>1$ are amphicheiral with $0$
as the unique non-trivial exceptional surgery, see the proof of Theorem 2 in Section 2.
As far as we know, $0$ is the only 
non-trivial exceptional surgery that occurs for an amphicheiral knot in $S^3$ that is not the figure-eight knot.

In \cite{MattmanCyclic} we showed
a weak form of Conjecture~\ref{Conj1} for cyclic surgery slopes and here we prove the analog for finite slopes. 
Boyer and Zhang~\cite{BoyerZhangFiniteFilling} showed that a finite surgery is integral or half-integral.

\begin{theorem}
\label{ThmFin}
If $t$ is a finite surgery on a hyperbolic knot $K$ in $S^3$ and $r_m$ and $r_M$ are 
the least and greatest finite boundary slopes of $K$, then $r_m - \frac52 \leq t \leq r_M + \frac52$.
If $t$ is half-integral, then  $r_m - 1 \leq t \leq r_M + 1$.
\end{theorem}

In Section 2, we prove Theorems~\ref{ThmAlt}, \ref{ThmMont}, and \ref{ThmTorti}.
In Section 3, we summarize our analysis of the 1267 knot complements in 
the census through nine tetrahedra and prove Theorem~\ref{Thm7tet}. 
We also give examples showing that each of the three parts of Conjecture~\ref{Conj5}
can occur and discuss Question~\ref{queConj53}.
In Section 4, we prove Theorem~\ref{ThmFin}.
Finally, the appendix includes tables of exceptional surgeries and relevant boundary slopes for the census knots.

\section{Proofs for Theorems~\ref{ThmAlt}, \ref{ThmMont}, and \ref{ThmTorti}}

In the following, $r_{(T)}$ (resp., $r_{(S)}$) indicates that the slope $r$ is a toroidal slope (resp., a Seifert slope). 

\begin{proof} (of Theorem~\ref{ThmAlt})
Let $K$ be a hyperbolic alternating knot in $S^3$ which admits 
non-trivial exceptional surgeries. 
Then, by \cite[Corollary 1.2.]{IchiharaMasai2016}, one of the following holds. 
\begin{itemize}
\item
$K$ is equivalent to the figure-eight knot, and the exceptional slopes are;
\[ -4_{(T)}, -3_{(S)}, -2_{(S)},-1_{(S)}, 0_{(T)}, 1_{(S)}, 2_{(S)}, 3_{(S)}, 4_{(T)}. \]
Thus (ii) in the statement of the theorem holds. 
\item 
$K$ is equivalent to a two bridge knot $K_{[2n, 2]}$ (resp. $K_{[2n, -2]}$) with $ | n | > 2$, and the exceptional slopes are;
\[ -4_{(T)}, -3_{(S)}, -2_{(S)},-1_{(S)}, 0_{(T)}.  \qquad (\text{resp. } 0_{(T)}, 1_{(S)}, 2_{(S)}, 3_{(S)}, 4_{(T)} .) \]
Thus (ii) in the statement of the theorem holds. 
\item
$K$ is equivalent to a two bridge knot $K_{[b_1,b_2]}$ with $ | b_1 | , | b_2 | > 2$ and both $b_1$ and $b_2$ even (resp. $b_1$ is odd and $b_2$ is even), and the exceptional slopes are $ 0_{(T)}$ (resp. $ 2 {b_2}_{(T)}$) only. 
Thus (i) in the statement of the theorem holds. 
\item
$K$ is equivalent to a pretzel knot $P(q_1,q_2,q_3)$ with $q_j \ne 0, \pm 1$ for $j = 1,2,3$ and $q_1,q_2,q_3$ are all odd (resp. $q_1$ is even and $q_2,q_3$ are odd), and the exceptional slopes are $ 0_{(T)}$ (resp. $2(q_2 + q_3)$) only. 
Thus (i) in the statement of the theorem holds. 
\end{itemize}
\end{proof}

In the following, $r_{(NI)}$ indicates that the slope $r$ is a non-integral boundary slope. 

\begin{proof} (of Theorem~\ref{ThmMont}).
Let $K$ be a hyperbolic Montesinos knot in $S^3$ admitting non-trivial exceptional surgeries. 
Then, by \cite[Theorem 3.6]{Wu1996}, the length of $K$ must be at most three. 
If $K$ is alternating, then Theorem~\ref{ThmAlt} assures the statement of the theorem holds for $K$. 
Thus, in the following, we assume that $K$ is non-alternating, which implies the length of $K$ is three, i.e., $K = M ( p_1/q_1, p_2/q_2, p_3/q_3 )$ with $q_j \ge 2$ for $j=1,2,3$. 
Also, since Montesinos knots have no reducible surgeries by \cite[Corollary 2.6]{Wu1996}, each exceptional surgery slope for $K$ is a toroidal slope or a Seifert slope. 
There are no toroidal Seifert slopes for Montesinos knots other than the trefoil \cite[Theorem 1.1]{IchiharaJong2010}. 

First, suppose that $K$ admits Seifert surgeries. 
Then, by \cite[Appendix B]{IchiharaMasai2016}, together with 
the program \cite{DunfieldProgram}, which can enumerate the boundary slopes for a given Montesinos knot based on the algorithm developed in \cite{HatcherOertel1989}, one of the following holds.

\begin{itemize}

\item
$K$ is equivalent to 
$P(-2, 3, 2n+1)$ with $n >3$, and the exceptional slopes and non-integral boundary slopes are;
\[ 
	4n+6_{(S)}, 4n+6+(1/(n-1))_{(NI)}, 4n+7_{(S)},	4n+8_{(T)}.
 \]
Thus (2) in Conjecture \ref{Conj5} holds. 
Further, there exists boundary slopes 0 and 16 for $K$, and so, Conjecture 1 also holds. 

\item
$K$ is equivalent to 
$P(-2, 3, 2n+1)$ with $n < -1$, and the exceptional slopes and non-integral boundary slopes are; 
\[ 
	4n+6+(2/(2n+1))_{(NI)}, 4n+6_{(S)}, 4n+7_{(S)},	4n+8_{(T)}.
 \]
Thus (2) in Conjecture \ref{Conj5} holds. 
As the exceptional surgeries are between
a NI boundary slope and a toroidal boundary slope, Conjecture 1 also holds. 

\item
$K$ is equivalent to 
$P(-2, 3, 7)$, and the exceptional slopes and non-integral boundary slopes are;
\[ 
	16_{(T)}, 17_{(S)}, 18_{(S)}, 18+1/2_{(T)}, 19_{(S)}, 20_{(T)}.
 \]
Thus (1) in Conjecture \ref{Conj5} holds. 

\item
$K$ is equivalent to 
$P(-3, 3, 3)$, and the exceptional slopes are;
\[ 
	0_{(T)}, 1_{(S)}, 2_{(T)}.
 \]
Thus (1) in Conjecture \ref{Conj5} holds. 

\item
$K$ is equivalent to 
$P(-3, 3, 4)$, and the exceptional slopes and non-integral boundary slopes are;
\[ 
	0_{(T)}, 1_{(S)}, 8/5_{(NI)}.
 \]
Thus (2) in Conjecture \ref{Conj5} holds. 

\item
$K$ is equivalent to 
$P(-3, 3, 5)$, and the exceptional slopes and non-integral boundary slopes are;
\[ 
	0_{(T)}, 1_{(S)}, 4/3_{(NI)}.
 \]
Thus (2) in Conjecture \ref{Conj5} holds. 

\item
$K$ is equivalent to 
$P(-3, 3, 6)$, and the exceptional slopes and non-integral boundary slopes are;
\[ 
	0_{(T)}, 1_{(S)}, 8/7_{(NI)}.
 \]
Thus (2) in Conjecture \ref{Conj5} holds. 

\item
$K$ is equivalent to 
$M(-1/2, 1/3, 2/5)$, and the exceptional slopes and non-integral boundary slopes are;
\[ 
	8/3_{(NI)}, 3_{(S)}, 4_{(S)}, 5_{(S)}, 6_{(T)}.
 \]
Thus (2) in Conjecture \ref{Conj5} holds. 

\item
$K$ is equivalent to 
$M(-1/2, 1/3, 2/7)$, and the exceptional slopes and non-integral boundary slopes are;
\[ 
	-2_{(T)}, -1_{(S)}, 0_{(S)}, 1_{(S)}, 3/2_{(NI)}.
 \]
Thus (2) in Conjecture \ref{Conj5} holds. 

\item
$K$ is equivalent to 
$M(-1/2,1/3,2/9)$, and the exceptional slopes and non-integral boundary slopes are;
\[ 
	3/2_{(NI)}, 2_{(S)}, 3_{(S)}, 4_{(S)}, 5_{(T)}.
 \]
Thus (2) in Conjecture \ref{Conj5} holds. 

\item
$K$ is equivalent to 
$M(-1/2, 1/3, 2/11)$, and the exceptional slopes and non-integral boundary slopes are;
\[ 
	-3_{(T)}, -2_{(S)}, -1_{(S)}, 0_{(T)}. 
 \]
Thus (1) in Conjecture \ref{Conj5} holds. 

\item
$K$ is equivalent to 
$M(-1/2, 1/5, 2/5)$, and the exceptional slopes and non-integral boundary slopes are;
\[ 
	32/5_{(NI)}, 7_{(S)}, 8_{(S)}, 9_{(T)}.
 \]
Thus (2) in Conjecture \ref{Conj5} holds. 

\item
$K$ is equivalent to 
$M(-1/2, 1/7, 2/5)$, and the exceptional slopes and non-integral boundary slopes are;
\[ 
	72/7_{(NI)}, 11_{(S)}, 12_{(T)}.
 \]
Thus (2) in Conjecture \ref{Conj5} holds. 

\item
$K$ is equivalent to 
$M(-2/3, 1/3, 2/5)$, and the exceptional slopes and non-integral boundary slopes are;
\[ 
	-6_{(T)}, -5_{(S)}, -4_{(T)}.
 \]
Thus (1) in Conjecture \ref{Conj5} holds. 

\end{itemize}

Next, we consider the case that $K$ has only toroidal slopes. 
If there is a single toroidal slope, then (1) in Conjecture \ref{Conj5} holds. 
The remaining cases are that $K$ has at least two toroidal slopes. 
Then, by \cite[(3) in page 308]{Wu2011}, $K$ admits exactly two toroidal surgeries, and $K$ is equivalent to one of five knots. 
Three of the five already appear in the above list. 
The remaining are the following. 

\begin{itemize}

\item
$K$ is equivalent to 
$P(-3,3,7)$
and the exceptional (i.e., toroidal) slopes are;
\[ 
	0_{(T)}, 1_{(T)}. 
 \]
Thus (1) in Conjecture \ref{Conj5} holds. 

\item
$K$ is equivalent to 
$M(-2/3, 1/3, 1/4)$, and the exceptional (i.e., toroidal) slopes are;
\[ 
	12_{(T)}, 13_{(T)}. 
 \]
Thus (1) in Conjecture \ref{Conj5} holds. 

\end{itemize}

This completes the proof. 

\end{proof}

Let $L_{\beta/\alpha} = K_1 \cup  K_2$ be a two-bridge link in $S^3$ associated to an irreducible fraction $\beta/\alpha$. 
After performing Dehn surgery on the component $K_2$ along the slope $-1/n$ for a positive integer $n$, i.e., after twisting $n$ times along $K_2$, the component $K_1$ becomes a knot in $S^3$. 
We call such a knot a \textit{torti-rational knot} and denote it by $K(\beta/\alpha;n)$. 
See \cite{Eudave-Munozetal2021, HirasawaMurasugi2010} for properties of torti-rational knots. 

We will show Theorem~\ref{ThmTorti}: certain hyperbolic torti-rational knots satisfy Conjectures \ref{Conj1} and \ref{Conj6}. 

In the following, let $[a_1,a_2, \dots, a_n]$ denote the continued fraction expansion of a rational number as follows. 
\[
[a_1,a_2, \dots, a_n]
= \frac{1}{a_1 + \frac{1}{a_2 + {\dots + \frac{1}{a_n}}}} 
\]
We follow the convention for two-bridge links used in \cite{Ichihara2012}, which is different (actually, the mirror image) of the standard 
one used in many knot theory textbooks. 
See Figure~\ref{L33} for example. 

\begin{figure}[htb]
 \centering
\includegraphics[width=.8\textwidth]{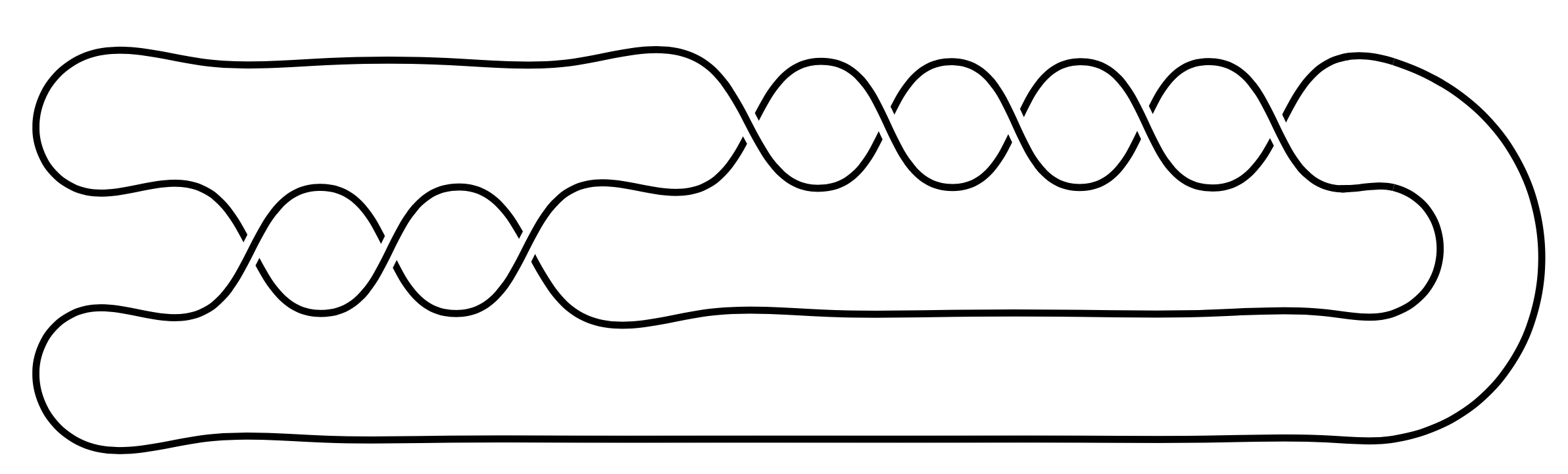}
\caption{$L_{[3,5]}$}\label{L33}
\end{figure}

For two-bridge links associated to continued fractions, the following holds by \cite[Theorem 2.1.11]{KawauchiBook}. 
(We remark that the two-bridge link in their notation is the mirror image of ours, but the result still holds.) 

\begin{lemma}\label{Lem1}
Let $[a_1, a_2, \dots, a_n]$ and $[b_1,b_2,\dots, b_m]$ be continued fractions such that all $a_i$'s and $b_j$'s are positive and none of $| a_1 | , | a_n | , | b_1 | , | b_m |$ is one. 
Then the corresponding two-bridge links $L_{[a_1, a_2, \dots, a_n]}$ and $L_{[b_1, b_2, \dots, b_m]}$ are equivalent (as unoriented links) by an orientation-preserving homeomorphism of $S^3$ if and only if $m = n$ and either $a_i = b_i$ or $a_i = \epsilon b_{m-i}$ with $\epsilon =(-1)^{m-1}$ for all $i$. 
\end{lemma}

We also remark that a two-bridge link $L_{[a_1, a_2, \dots, a_n]}$ is equivalent to $L_{[a_n, a_{n-1}, \dots, a_1]}$ if $n$ is odd (respectively, $L_{[-a_n, -a_{n-1}, \dots, -a_1]}$ if $n$ is even) by an orientation-preserving homeomorphism of $S^3$.

\bigskip

Now Theorem~\ref{ThmTorti} follows from the following two lemmas. 

\begin{lemma}\label{lemTR1}
Let $K = K(\beta/\alpha;n)$ be a hyperbolic torti-rational knot with $n \ge 4$. 
If a slope $r$ is an exceptional slope for $K$, then the slope $r' = r - n l^2$ is a Seifert slope or a toroidal slope for one component of the hyperbolic two-bridge link $L_{\beta/\alpha} = K_1 \cup K_2$, where $l = lk (K_1, K_2)$. 
If a slope $r'$ is a Seifert slope (respectively, a toroidal slope) for one component of a hyperbolic two-bridge link $L_{\beta/\alpha} = K_1 \cup K_2$, then $r = r' + n l^2$ is a Seifert slope (resp., toroidal slope) for $K = K(\beta/\alpha;n)$, where $l = lk (K_1, K_2)$. 
\end{lemma}

\begin{proof}
For $\beta/\alpha$ and $n \ge 4$, let $K = K(\beta/\alpha;n)$ be the torti-rational knot and $L_{\beta/\alpha} = K_1 \cup K_2$ the two-bridge link with $l = lk (K_1, K_2)$.
We assume both are hyperbolic. 

Given a slope $r$ for $K$,
we note that the manifold $K(r)$ is also obtained by Dehn surgery on the two-bridge link $L_{\beta/\alpha}$ along the slope $(r', -1/n)$. That is, $K(r) \cong L_{\beta/\alpha} (r',-1/n)$, where $r' = r - n l^2$. 
For calculation of surgery slopes, see \cite{RolfsenBook} for example.

Suppose first that a slope $r$ is an exceptional slope for $K$.
Then, by the result of \cite{IchiharaJongMasai2019} together with the assumption $n \ge 4$, if 1) $L_{\beta/\alpha} = K_1 \cup K_2$ is hyperbolic, 2) $K$ is hyperbolic, and 3) 
$K(r) \cong L_{\beta/\alpha} (r',-1/n)$ is non-hyperbolic; then $L_{\beta/\alpha} (r',*)$ is also non-hyperbolic, where $*$ means that $K_2$ remains unfilled and $r' = r - n l^2$. 
That is, the slope $r'$ is exceptional for a component of $L_{\beta/\alpha} = K_1 \cup K_2$. 
Such an exceptional slope must be a Seifert slope or a toroidal slope by \cite[Theorem 1.1]{Ichihara2012}. 

Suppose next that a slope $r'$ is a Seifert slope or a toroidal slope for one component, say $K_1$, of a hyperbolic two-bridge link $L_{\beta/\alpha} = K_1 \cup K_2$.
As noted above,  $L_{\beta/\alpha} (r',-1/n) \cong K(r) $ holds, where $r = r' + n l^2$. 

If $r'$ is a Seifert slope, then, by the proof of \cite[Theorem 1.1]{Ichihara2012}, $K_2$ becomes 
a non-trivial non-core torus knot $K'$ in a lens space after $r'$-surgery on $K_1$. 
This shows that $K(r)$ is obtained by some non-integral surgery on $K'$ and
is a small Seifert fibered manifold.

If $r'$ is a toroidal slope, then, by the proof of \cite[Theorem 1.1]{Ichihara2012}, $K_2$ becomes 
a non-trivial satellite knot $K'$ in a lens space after $r'$-surgery on $K_1$. 
Moreover the companion knot of $K_2$ is a torus knot in the lens space, and the pattern of $K'$ 
is either hyperbolic or a $(2,k)$-cable for some integer $k$. 
If the pattern of $K$ is a $(2,k)$-cable for some $k$, then 
$-1/n$-surgery 
on $K'$ is toroidal for $n \ge 4$ by \cite[Lemma 7.2]{Gordon83}. 
On the other hand,
if the pattern of $K$ is hyperbolic, then 
$-1/n$-surgery 
on $K'$ is toroidal for $n \ge 4$ by \cite[Theorem 2.0.1]{CGLS}, since $1/0$-surgery on $K'$ compresses the satellite torus for $K'$. 
\end{proof}

\begin{lemma}\label{lemTR2}
The exceptional slopes of a component, $K$, of a 
hyperbolic two-bridge link $L$ occur in one of two ways.
Either 
(i) there is a toroidal boundary slope which is the only non-trivial exceptional slope for $K$, or else 
(ii) there are a pair of integral toroidal boundary slopes, and the non-trivial exceptional slopes for $K$ are the integers in the interval they bound.
\end{lemma}

\begin{proof}
Let $L_{\beta/\alpha} = K_1 \cup K_2$ be a hyperbolic two-bridge link. 

First suppose that there is a Seifert slope $r$ for one component, say $K_1$, of $L_{\beta/\alpha}$. 
Then, by \cite[Theorem 1.1]{Ichihara2012}, the link $L_{\beta/\alpha}$ is equivalent to $L_{[2w+1,2u+1]}$ for $w \ge 1$, $u \ne  0,-1$. 
There are several cases\footnote{There is a typo in the statement of \cite[Theorem 1.1]{Ichihara2012}. 
There, (b) should be $L_{[2w+1,-3]}$ since the proof of Claim 4 in the paper, which states the corresponding part of the theorem, depends essentially on \cite[Theorem 11.1(2)]{GodaHayashiSong2009}.
That theorem says that $L([2w + 1, -3])[-w - 1]$ gives a torus knot in a lens space.
However, since $L([2w + 1, -3])[-w - 1]$ is homeomorphic to $L([3, 2(-w-1) + 1)])[-w - 1]$, case (b) is actually included in case (a). Thus we will
ignore it in this paper.} 
depending on the values of $w$ and $u$.

If $w \ge 2$ and $u \ne 1,0,-1,-2$, then $L_{[2w+1,2u+1]}$ admits a
single Seifert slope $r = -w + u$ on $K_1$ by \cite[Theorem 1.1]{Ichihara2012}. 

For the case of $w \ge 2$ and $u \ge 2$, 
by calculating the
continued fractions, we see that $[2w+1,2u+1] = [2w,1,2(-u-1)] = [2(w+1),-1,-2u]$. 
Then, again by \cite[Theorem 1.1]{Ichihara2012}, the slopes $ -w - (-u-1) = -w+u+1$ and $ -(w+1) - (-u) = -w+u-1$ are both toroidal slopes on $K_1 \subset L_{[2w+1,2u+1]}$. 
That is, $L_{\beta/\alpha} (-w+u+1,*)$ and $L_{\beta/\alpha} (-w+u-1,*)$ 
both contain essential tori.

To complete the proof for the case of $w \ge 2$ and $u \ge 2$, we need to show that there are no other exceptional (i.e., toroidal) slopes for $K_1 \subset L_{[2w+1,2u+1]}$ in this case. 

We use the following claim that
can be shown by direct calculation of continued fraction expansions; we omit the details.

\begin{claim}\label{Clm1}
Suppose that $w = 1, u = -1, |v| \ge 2$ or $w \ge 2, |u| \ge 2, |v| \ge 1$.
Then the simple continued fraction (i.e., all terms positive) for $[2w,v,2u]$ is as follows. 
\begin{align}
&[2w,v,2u] \ 
& \mbox{ if } v,u >0 \\
&[2w,1,2u]= [2w+1, -2u-1]  & \mbox{ if } v=1,u \le -2 \\
&[2w,v,-2]= [2, v-1,2]   & \mbox{ if } v \ge 2, u =-1, (w=1)\\
&[2w,v-1,1,-2u-1] & \mbox{ if } v \ge 2, u \le -2 \\
&[2w,-1,2u] = [2w-2,1,2u-2] & \mbox{ if } v = -1, u \ge 2 , (w \ge 2) \\
&[2w,-2,2u]=[2w-1,2,2u-1] & \mbox{ if } v = -2, u \ge 2, (w \ge 2)  \\
&[2w-1,1,-v-2,1,2u-1] & \mbox{ if } v \le -3, u \ge 2 , (w \ge 2) \\
&[2w,-1,2u] = [2w-1,-2u+1] & \mbox{ if } v = -1, u \le -2 , (w \ge 2) \\
&[1,1,-v-1,2] & \mbox{ if } v \le -2, u = -1, (w=1)\\
&[2w-1,1,-v-1,-2u] & \mbox{ if } v \le -2, u \le -2, (w \ge 2) 
\end{align}
\end{claim}

Assume for a contradiction that $K_1$ 
admits some other toroidal surgery. 
This implies that $L_{[2w+1,2u+1]}$ with $w \ge 2$ and $u \ge 2$ admits a toroidal slope on 
one of its components.
Then, by \cite[Theorem 1.1]{Ichihara2012}, $L_{[2w+1,2u+1]}$ is equivalent to $L_{[2w',v',2u']}$
with $w' = 1$, $u' = -1$, $|v'| \ge 2$, or $w' \ge 2$, $|u'| \ge 2$, $|v'| \ge 1$. 
As already seen, $[2w+1,2u+1] = [2w,1,2(-u-1)] = [2(w+1),-1,-2u]$ holds, and so, $(w',v',u')=(w,1,-u-1)$ (Claim~\ref{Clm1} (1)) or $(w',v',u')=(w+1,-1,-u)$ (Claim~\ref{Clm1} (8)) actually happen, but no other cases can happen by Claim~\ref{Clm1} together with Lemma~\ref{Lem1}, except for the case $(w',v',u')=(1,v,-1)$ with $v \le -2$ (Claim~\ref{Clm1} (9)). 
(In this case, Lemma~\ref{Lem1} cannot be applied directly, for the first term of $[1,1,-v-1,2]$ is 1.) 
For this case, we see that 
\begin{align*}
L_{[2w+1,2u+1]} & \cong L_{[-2w-1,-2u-1]}  \\
& \not\cong L_{[-2,v+1,-2]} 
\cong L_{[-2,v+1,-1,-1]} 
\cong L_{[1,1,-v-1,2]} \cong L_{[2,v,-2]}
\end{align*}
by Lemma~\ref{Lem1}, where $\cong$ denotes the equivalence of knots by orientation-preserving homeomorphism of $S^3$. 
This implies that $L_{[2w+1,2u+1]}$ with $w \ge 2$ and $u \ge 2$ only admits toroidal slopes $-w+u-1$ and $-w+u+1$ on $K_1$. 
That is, the exceptional slopes for a component of the link are;
\[
-w+u-1_{(T)} , -w+u_{(S)} , -w+u+1_{(T)}. 
\]

For the case of $w \ge 2$ and $u \le -3$, we can argue similarly. 
If $L_{[2w+1,2u+1]}$ with $w \ge 2$ and $u \le -3$ 
admits a toroidal slope on a component, $K$, 
then $L_{[2w+1,2u+1]}$ is equivalent to $L_{[2w',v',2u']}$
with $w' = 1$, $u' = -1$, $|v'| \ge 2$, or $w' \ge 2$, $|u'| \ge 2$, $|v'| \ge 1$ by \cite[Theorem 1.1]{Ichihara2012}. 
Since $[2w+1,2u+1]$ has the simple continued fraction $[2w,1,-2u-2]$, the possible cases are only $(w',v',u') =(w,1,-u-1)$  (Claim~\ref{Clm1} (1)) or $(w',v',u') =(w+1,-1,-u)$  (Claim~\ref{Clm1} (5)) by 
Claim~\ref{Clm1}, Lemma~\ref{Lem1}, and the above argument. 
In these cases, by \cite[Theorem 1.1]{Ichihara2012}, $ -w' -u' = -w - (-u-1) = -w+u+1$ and $ -w'-u'= -(w+1) - (-u) = -w+u-1$ are both toroidal slopes on the component $K$ 
of $L_{[2w+1,2u+1]}$ with $w \ge 2$ and $u \le -3$. 
This implies that $L_{[2w+1,2u+1]}$ with $w \ge 2$ and $u \le -3$ only admits toroidal slopes $-w+u-1$ and $-w+u+1$ on 
the component $K$. 
That is, the exceptional slopes for a component of the link are;
\[
-w+u-1_{(T)} , -w+u_{(S)} , -w+u+1_{(T)}. 
\]

There remain the cases where $w=1$ or $u=1$ or $u=-2$. 
In fact, these can be treated simultaneously, for 
\[ L_{[3,2u+1]} \cong L_{[-2u-1,-3]} \cong ( L_{[2w+1,3]} )^* \]
holds by setting $u=w$, and 
\[ L_{[3,2u+1]} \cong L_{[-2u-1,-3]} \cong L_{[2w+1,-3]} \]
holds by setting $u=-w-1$. 
Thus we only consider the link $L_{[3,2u+1]}$ with $u \ne 0,-1$. 

Although we can apply the same arguments as above, 
in this case, we may use a simpler argument
due to \cite{MartelliPetronio2006}. 
That's because, 
in this case, $L_{[3,2u+1]}(r,*)$ is homeomorphic to $N(-\frac{1}{u+1},r-u-1)$, where $N$ is the exterior of the minimally twisted 3-chain link, that is, the so-called
``magic manifold'' \cite{MartelliPetronio2006}. 
See Figure~\ref{L3MT3C}.

\begin{figure}[htb]
\centering
\includegraphics[width=.95\textwidth]{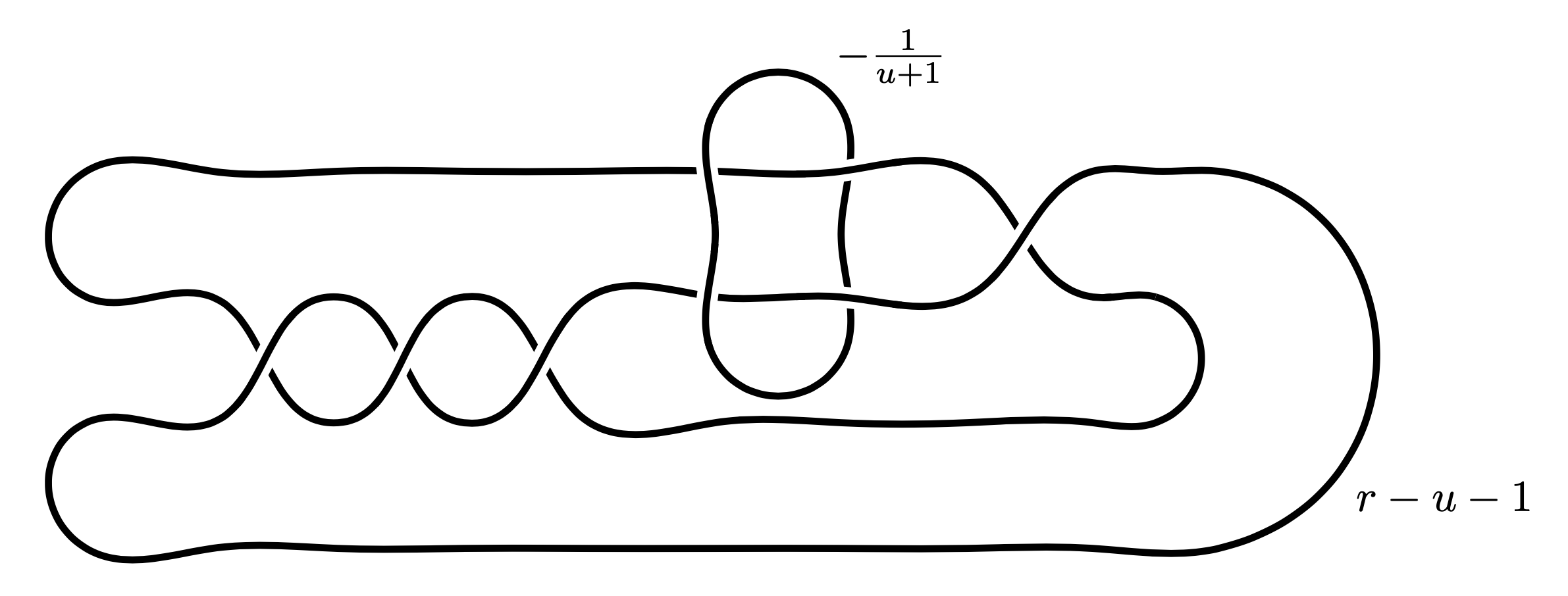}
\caption{$L_{[3,2u+1]}(r,*) \cong N(-\frac{1}{u+1},r-u-1) $}\label{L3MT3C}
\end{figure}

Suppose that $L_{[3,2u+1]}(r,*)$ is non-hyperbolic and $L_{[3,2u+1]}$ is hyperbolic. 
Note that $u \ne 0,-1$ by the assumption that $L_{[3,2u+1]}$ is hyperbolic. 
It follows that $-\frac{1}{u+1}$ is not an integer. 
Then, by \cite[Theorems 1, 2]{MartelliPetronio2006}, we have the following cases. 

If $ r-u-1 \in \{ -3,-2,-1,0 \}$, i.e., $ r = u-2,u-1,u,$ or $u+1$, then $L_{[3,2u+1]}(r,*)$ is non-hyperbolic. 
Moreover, by \cite[Tables 1, 2]{MartelliPetronio2006}, if $u \ne -2,1$, the exceptional slopes for a component of $L_{[3,2u+1]}$ are;
\[
u-2_{(T)} , u-1_{(S)} , u_{(S)}, u+1_{(T)}. 
\]

In the case of $u=-2$, if $L_{[3,-3]}(r,*)$ is non-hyperbolic, then $r=0$ can occur, in addition to $ r \in \{ -4,-3,-2,-1 \}$. 
Thus, by \cite[Tables 1, 2]{MartelliPetronio2006}, the exceptional slopes for a component of $L_{[3,-3]}$ are;
\[
-4_{(T)} , -3_{(S)} , -2_{(S)}, -1_{(S)}, 0_{(T)}. 
\]
Actually $L_{[3,-3]}$ is the Whitehead link, 
for which these results 
are well-known. 
See \cite[Tables A.1]{MartelliPetronio2006}. 

In the case of $u=1$, if $L_{[3,3]}(r,*)$ is non-hyperbolic, then $r=-2$ can occur, in addition to $ r \in \{ -1,0,1,2 \}$. 
Thus, by \cite[Tables 1, 2]{MartelliPetronio2006}, the exceptional slopes for a component of $L_{[3,3]}$ are;
\[
-2_{(T)} , -1_{(S)} , 0_{(S)}, 1_{(S)}, 2_{(T)}. 
\]
Actually $L_{[3,3]} = L_{3/10}$ is also 
a well-known link. 
See \cite[Tables A.1]{MartelliPetronio2006}. 

\bigskip

Next suppose that there are only toroidal slopes for one component, say $K_1$, of $L_{\beta/\alpha}$. 
We will show that $K_1$ admits only one toroidal slope.  

In this case, by \cite[Theorem 1.1]{Ichihara2012}, the link $L_{\beta/\alpha}$ is equivalent to $L_{[2w,v,2u]}$ with $w = 1$, $u = -1$, $|v| \ge 2$, or $w \ge 2$, $|u| \ge 2$, $|v| \ge 1$. 
And suppose further that the knot admits another toroidal slope. 
Then, in the same way, $L_{\beta/\alpha}$ is equivalent to $L_{[2w',v',2u']}$
with $w' = 1$, $u' = -1$, $|v'| \ge 2$, or $w' \ge 2$, $|u'| \ge 2$, $|v'| \ge 1$, and $(w,v,u) \ne (w',v',u')$. 
Now, by Claim~\ref{Clm1} together with 
Lemma~\ref{Lem1}, it suffices
to consider the 10 cases each for $(w,v,u)$ and $(w',v',u')$. 
However, by Claim~\ref{Clm1} together 
with Lemma~\ref{Lem1}, it is unnecessary to consider cases where 
the lengths of their simple continued fraction expansions are different unless $v \le -2, u = -1, w = 1$  (Claim~\ref{Clm1} (9)). 

We look at a limited number of cases in the following. The remaining cases can be shown in the same way. 

Consider the case that $u,v>0$. 
Then $[2w,v,2u]$ is a simple continued fraction. 
If the simple continued fraction associated to $[2w',v',2u']$ is coincident to $[2w,v,2u]$, then 
we have $(w',v',u')=(w+1,-1,u+1)$ and $v=1$ (Claim~\ref{Clm1} (2)). 
In this case, $[2w,1,2u]=[2w+1,-2u-1]$ also holds. 
This implies that $L_{[2w,1,2u]} \cong L_{[2(w+1),-1,2(u+1)]} \cong L_{[2w+1,-2u-1]}$ admits a Seifert slope $r=-w+(-u-1)=-w-u-1$ on one component, contradicting that $K_1$ has only toroidal surgeries. 
The cases $v=1, u \le -2$, $v =-1, u \ge 2$ and $v = -1, u \le -2$ are similar. 

Consider the case that $v \ge 2, u =-1, (w=1)$. 
Then $[2w,v,2u]$ has the simple continued fraction $[2,v-1,2]$. 
Under the conditions for $(w',v',u')$, there are no $[2w',v',2u']$ with the simple continued fraction $[2,v-1,2]$ by Claim~\ref{Clm1}. 
The cases $v =-2, u \ge 2$, $v \ge 2, u\le -2$ and $v \le -3, u \ge 2$ are similar. 

Consider the case that $v \le -2, u = -1, (w = 1)$. 
The toroidal slope for a component of $L_{[2,v,-2]}$ is $-1-(-1)=0$. 
Then $[2w,v,2u]=[2,v,-2]$ ($v \le -2$) has the simple continued fraction $[1, 1, -v-1, 2]$  (Claim~\ref{Clm1} (9)). 
\[ L_{[2,v,-2]} \cong L_{[1,1,-v-1,2]} \cong L_{[-2,v+1,-1,-1]} \cong L_{[-2,v+1,-2]}\]
Thus the link is the mirror image of $L_{[2,-v-1,2]}$ with $v \le -2$. 
Under the conditions for $(w',v',u')$, there are no $[2w',v',2u']$ with the simple continued fraction $[2,-v-1,2]$ by Claim~\ref{Clm1}. 

Consequently, if there are only toroidal slopes for a component of $L_{\beta/\alpha}$, 
then, in fact, there is a single toroidal slope. 
\end{proof}

\section{Census knots
\label{SecCensus}%
}

In this section we provide evidence for Conjectures~\ref{Conj1} and \ref{Conj6} using Dunfield's census of 
exceptional Dehn fillings~\cite{DunfieldCensus} on knots with complements that have at most nine tetrahedra.
We prove Theorem~\ref{Thm7tet} that Conjectures~\ref{Conj1} and \ref{Conj6} hold for knots 
of at most seven tetrahedra. We give examples that realize each of the three parts of 
Conjecture~\ref{Conj5} and ask if knots that satisfy part (3) of the conjecture
can always do so by way of a single non-integral boundary slope $b_1 = b_2 = b$.

Dunfield's census includes the complements of 1267 
hyperbolic knots in $S^3$.
Of these, 165 have no non-trivial exceptional surgeries and 418 are such that all non-trivial exceptional surgeries are toroidal. 
In each case, the toroidal surgeries occur as a list of consecutive integers of length between one and three and, 
therefore, both conjectures are verified for these knots.

The remaining knot complements in the census have at
least one non-trivial, non-toroidal exceptional surgery.
For those with at most seven tetrahedra, the $A$-polynomial calculations of Culler~\cite{CullerApoly} provide a substantially complete list of boundary slopes.
(There's no guarantee that the $A$-polynomial 
detects all boundary slopes. However, in practice, it does seem to find almost all of them; for example see \cite{Segerman} for a detailed account.) 
This allows us to prove Theorem~\ref{Thm7tet}.
(Although it would be enough to rely on the $A$-polynomial calculations, 
in the data presented in the appendix, 
for the sake of convenience
we also used boundary slopes found using Dunfield's program
for two-bridge and Montesinos knots~\cite{DunfieldProgram} and `Kabaya' boundary
slopes found with SnapPy~\cite{SnapPy}.)

\begin{proof}
(of Theorem~\ref{Thm7tet})
Among the census manifolds with at
most seven tetrahedra, there are 201 that are complements of hyperbolic knots in $S^3$.
In the appendix we list exceptional 
and boundary slopes for each of these knot complements, 
which show that both Conjectures~\ref{Conj1} and 
\ref{Conj6} hold. 
\end{proof}

Note that the complements of hyperbolic
knots in $S^3$ with at most seven tetrahedra have names that begin with `m', `s', or `v'.
The appendix lists all toroidal boundary slopes, but only enough of the other boundary slopes
to verify the conjectures. Let us give the details for two examples as an illustration.

\begin{align*}
\mbox{v0319} \quad & 
[(-62, \mbox{`T'}), -63, -64, (-194/3, \mbox{`C'}), -65, (-206/3, \mbox{`C'})] \\
& [(-2, \mbox{`T'}), -1, 0, (2/3, 
  \mbox{`C'}), 1, (14/3, \mbox{`C'})] \\
 \\
\mbox{v1359} \quad &
[(121/2, \mbox{`CK'}), 59, (176/3, \mbox{`C'}), 58, (57, \mbox{`T'})] \\
& [(-5/2, \mbox{`CK'}), -1, (-2/3, \mbox{`C'}), 0, (1, \mbox{`T'})]
\end{align*}

Using Dunfield's census and SnapPy coordinates, the exceptional slopes of v0319 are $-2,-1,0,1$, with
$-2$ toroidal. For Conjecture~\ref{Conj1}, use $b_1 = -2$
and $b_2 = 14/3$, a slope that we identified using 
Culler's~\cite{CullerApoly} calculation of the $A$-polynomial. For Conjecture~\ref{Conj6}, 
use $b_1 = -2$ and $b_2 = 2/3$, another slope found using the $A$-polynomial. This knot has
other boundary slopes besides $-2$, $2/3$, and $14/3$. Here we only mention slopes that are toroidal or useful for verifying the conjectures. For v1359, the exceptional slopes
are, $-1,0,1$, and we have $b_2 = 1$, 
a toroidal slope, for both conjectures. For Conjecture~\ref{Conj1}, $b_1 = -5/2$, a slope
identified both using Culler's~\cite{CullerApoly} $A$-polynomial calculation and as a Kabaya boundary slope
by SnapPy~\cite{SnapPy}. For Conjecture~\ref{Conj6}, $b_1 = -2/3$, a slope
found using the $A$-polynomial.

Similar to Theorem~\ref{Thm7tet}, for two bridge knots, which are alternating, or Montesinos knots, we 
know the boundary slopes by
\cite{DunfieldProgram}
\cite{HatcherOertel1989}, and \cite{HatcherThurston}
and we can verify both conjectures for these knots as in Theorems~\ref{ThmAlt} and \ref{ThmMont}.
In other words, in all cases where we have reasonably
complete information about the boundary slopes, we
can verify both conjectures.

Finally, there are many examples of knots where we can verify
our two conjectures directly as the set of nontrivial exceptional slopes
is bounded between two toroidal boundary slopes.

There remain the 336 knot complements that have: a non-trivial, non-toroidal exceptional surgery;
a triangulation with eight or nine tetrahedra;
are not two bridge or Montesinos; and either the greatest or
the least exceptional slope is not toroidal.
We have verified both conjectures for many of these knots. However there remain 75 (respectively, 166) knot complements for which we have not been able to verify Conjecture~\ref{Conj1}
(resp., Conjecture~\ref{Conj6}). Of the 75 outstanding knots for Conjecture~\ref{Conj1} more
than twenty satisfy Conjecture~\ref{Conj6}, 
leaving 51 knot complements for which we can verify neither conjecture.

While there remain many outstanding knot complements for each conjecture, we expect that this mainly reflects our ignorance
about the boundary slopes for these knots. To reiterate, in every case 
where we have substantial knowledge about the boundary slopes, we were
able to verify both conjectures.

Let's observe that each part of Conjecture~\ref{Conj5} does, in fact, arise.
As mentioned in the introduction, two bridge knots $K_{[a_1,a_2]}$ with $|a_1|,|a_2| > 2$
and pretzel knots $P(q_1,q_2,q_3)$ with $q_j \neq 0,\pm 1$ for $j = 1,2,3$ are examples
where the only non-trivial exceptional surgery $b$ is toroidal. This means we can 
choose $b_1 = b_2 = b$ for these knots, as in Conjecture~\ref{Conj5} (1). 
Examples that realize (1) with 
with $b_1 < b_2$, both toroidal, include
the twist knots $K_{[2n,2]}$ and $K_{[2n,-2]}$ with $|n| > 2$.
The pretzel knots $P(-2,3,2n+1)$ with $n > 3$ or $n < -1$ provide examples
for Conjecture~\ref{Conj5} (2) with $b_1$ non-integral, $b_2$ toroidal, and $b_1 < b_2$.

Examples that realize Conjecture~\ref{Conj5} (3) include s682,
s769, v1300, v1628, v1940, v1966, v2217, v2759, v2871, v2925, and v3234. Of these,
seven knots: 
s682, v1300, v1628, v2217, v2759, v2871, and v3234 {\em could} be explained with a 
choice of non-integral boundary slopes $b_1 < b_2$. However, these seven also 
satisfy Conjecture~\ref{Conj5} (3) thanks to a single non-integral slope $b$ with
$b_1 = b_2 = b$. For example, s682 has exceptional slopes $-1,0$ and boundary slopes $-3/2$ and $-1/3$. It satisfies
Conjecture~\ref{Conj5} both with the choice $b_1 = -3/2$, $b_2 = -1/3$, 
as well as with $b_1 = b_2 = -1/3$.
All the examples that we know of for Conjecture~\ref{Conj5} (3) can be shown
to satisfy the conjecture by the choice of a single non-integral boundary slope 
$b_1 = b_2 = b$. This is the source of Question~\ref{queConj53} 
mentioned in the introduction.

In the appendix we provide a listing of the 1267 census knots with their exceptional slopes and relevant boundary slopes.
The exceptional slopes come from Dunfield's census. There are several sources for boundary slopes: 
toroidal slopes, which are exceptional slopes, 
from Dunfield's census;
boundary slopes deduced from Culler's A-polynomial calculations~\cite{CullerApoly};
boundary slopes of two bridge~\cite{HatcherThurston} or
Montesinos knots~\cite{HatcherOertel1989};
Brasile and Kabaya boundary
slopes as determined by SnapPy~\cite{Brasile, SnapPy, Kabaya}; and the
longitude, which is the boundary of a Seifert surface.
We thank Nathan Dunfield for helpful conversations 
and for telling us about Brasile boundary slopes.

\section{Finite slopes}

In this section we prove Theorem~\ref{ThmFin}.

Let $r_m$ (respectively, $r_M$) denote the smallest (resp., largest) finite boundary slopes, considered as rational numbers. Conjecture~\ref{Conj1} states that all non-trivial exceptional
surgeries must occur in the interval $[r_m, r_M]$. As progress towards this conjecture, the second author~\cite{MattmanCyclic} showed
that all cyclic slopes occur in the interval $(r_m- \frac12, r_M + \frac12)$. Here, we use a similar approach to show that all finite slopes occur in the interval
$[r_m - \frac52, r_M + \frac52]$. 

The argument is similar to that in~\cite{MattmanCyclic}
and leverages properties of the Culler-Shalen norm.
We refer the reader to that paper for a brief introduction
to the relevant properties. See~\cite{CGLS} or \cite{Boyer}
for a more detailed overview.

A finite slope is either an integer $n$ or a half integer $\frac{n}{2}$~\cite{BoyerZhangFiniteFilling}. In their proof of the Finite Filling Conjecture, 
Boyer and Zhang~\cite{BoyerZhangFiniteFilling} show
that the Culler-Shalen norm of a finite slope is bounded by $\max (2s, s+8)$ where $s$ represents the minimal norm. 
For a hyperbolic knot,  $s \geq 4$~\cite[Lemma 9.1]{BoyerZhangFiniteFilling}, so that $3s$ is always a bound for the norm of a finite filling slope.

Note that it follows from~\cite{IMS}, that a finite 
surgery $n \in \Z$ is in $[r_m - 3, r_M + 3]$, 
while a half-integer slope $\frac{n}{2}$ is
in $[r_m - \frac32, r_M + \frac32]$, see \cite[Theorem 1]{MattmanCyclic}. 

\begin{proof} (of Theorem~\ref{ThmFin})
We split into two cases depending on whether the finite slope is integer or half-integer.
Suppose first that the positive integer $n$ is a finite surgery slope and $\| (n,1) \| = t \leq 3s$. 
As in~\cite{MattmanCyclic}, the argument comes down to showing $n \leq r_M+\frac52$. 
For a contradiction, suppose instead $n > r_M+\frac52$.

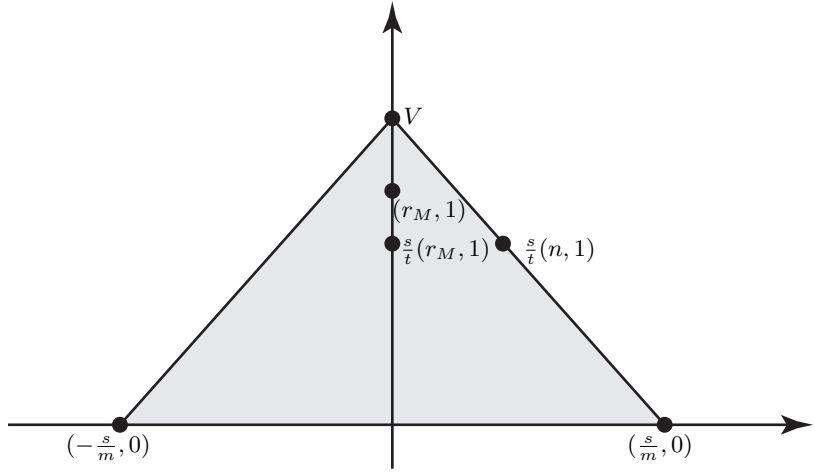
\begin{figure}[htb]
\labellist{}
\small\hair 2pt
\pinlabel {$( - \frac{s}{m}, 0)$} at 42 8
\pinlabel {$( \frac{s}{m} , 0)$} at 250 8
\pinlabel $\frac{s}{t}(r_M,1)$ at 169 82
\pinlabel $\frac{s}{t}(n,1)$ at 212 82
\pinlabel $(r_M,1)$ at 163 97
\pinlabel $V$ at 157 133
\endlabellist{}
\centering
\includegraphics[scale=1]{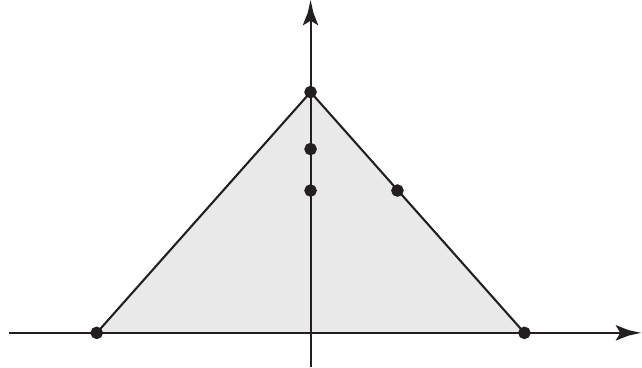}
\caption{The triangle $T$ is shaded.
\label{figT}%
}
\end{figure}

Let $P$ denote the fundamental polygon of the 
Culler-Shalen norm, the set of points of minimal norm $s$. 
Let $\| \mu \| = \| (1,0) \| = m$. Then $ \pm (s/m, 0)$ are points
on the boundary of the fundamental polygon $P$ as is the point $\frac{s}{t}(n,1)$. Choose coordinates so that $(r_M,1)$ appears on the $y$-axis. Extend the
line through $(s/m,0)$ and $\frac{s}{t}(n,1)$ until it intersects the $y$-axis at $V$, which will be a vertex of $P$. 
(See Figure~\ref{figT}. As in \cite{MattmanCyclic}, we can assume that $P$
has a vertex that corresponds to $r_M$.) By convexity, the
segment $(-s/m, 0)$ to $V$ is in $P$ as is the isosceles triangle $T$ with vertices $\pm (s/m,0)$ and $V$. We will argue that the intersection of $T$ 
with the line $y = 1$ is a segment whose length exceeds one. This implies there is at least one lattice point interior to $T \subset P$, which is a contradiction.
Note that if $r_M$ is an integer, then $(r_M,1)$ is itself a lattice point interior to $P$.

Let $w(y)$ denote the width of $T$ at height $y$. 
Note that $w(0) = 2s/m$ and $w(s/t) = 2 \frac{s}{t}( n - r_M)$. 
Then, since $w(y)$ is a linear function of $y$,
\begin{align*}
w(1) = & \frac{2s}{m} - \frac{\frac{2s}{m} - \frac{2s}{t} ( n - r_M) }{\frac{s}{t}} \\
= & \frac{2s}{m}(1 - \frac{t}{s}) + 2 (n - r_M) \\
= & 2(n-r_M) - \frac{2}{m} (t-s),
\end{align*}
which will be larger than 1. Indeed, $s \leq t \leq 3s$ and $s \leq m$ imply that $\frac{2}{m}(t-s)$ is at most four
and we are assuming $2(n - r_M)$ is larger than five. Since $P$ is the polygon defined by the minimal norm $s$, an
interior lattice point would be a slope of norm strictly less than the minimal norm, which is absurd.

Next assume $\frac{n}{2}$ is a finite slope so that $\| (n,2) \| = t \leq 3s$. The argument is similar where we replace $\frac{s}{t} (n,1)$ 
as a point on the polygon with $\frac{s}{t}(n,2)$. However, in this case we can argue that $\frac{n}{2} \leq r_M + 1$. Assume, for a contradiction 
that $\frac{n}{2} > r_M + 1$. Again $w(0) = 2s/m$ and $w(2s/t) =  \frac{2s}{t}(n - r_M)$. Then, 
\begin{align*}
w(1) = & \frac{2s}{m} - \frac{\frac{2s}{m} - \frac{2s}{t} ( n - 2r_M) }{\frac{2s}{t}} \\
= & \frac{2s}{m}(1 - \frac{t}{2s}) + (n - 2r_M) \\
= & (n-2r_M) - \frac{t-2s}{m}.
\end{align*}
Since $s \leq t \leq 3s$ and $s \leq m$, then $\frac{t-2s}{m}$ is at most one. But $\frac{n}{2} - r_M > 1$ means
there will be a lattice point inside the fundamental polygon, which is a contradiction.
\end{proof}

\section*{Acknowledgements}
We very much appreciate Professor Motegi's mentorship and guidance over the years, including providing the
question that is the inspiration for this project. We thank Nathan Dunfield for helpful conversations, for telling us 
about Brasile boundary slopes, and for his wonderful census of knot manifolds through nine tetrahedra.
We are grateful to the referee for an encouraging assessment of our paper and useful suggestions.

\section{Appendix: Slopes of Census Knots}
In the appendix we list exceptional and boundary slopes for the
1267 hyperbolic knots in $S^3$ included in 
Dunfield's~\cite{DunfieldCensus} census. We provide the data
in three \verb".csv" files. 

Among these knot complements, 165 have no non-trivial exceptional surgeries and 418 are such that all non-trivial exceptional surgeries are toroidal. 
These 583 knots are gathered in the file
\verb"TorOnly.csv". 
For each knot, we list the census name, the 
(possibly empty) list of toroidal slopes both in standard and
SnapPy coordinates (as reported in Dunfield's census), 
and the name of the knot, if we know it.
We remark that Baker and 
Kegel~\cite{BakerKegelpaper, BakerKegeldata} have determined 
braid words for every one of the 1267 census knots.
As discussed in Section~\ref{SecCensus}, 
these knots satisfy both Conjecture~\ref{Conj1} and \ref{Conj6}.

The remaining census knots each have at least one non-trivial
non-toroidal exceptional slope. 
Of these, the first group of 348 knots presented in the file 
\verb"Verified348.csv", 
are those for which we can easily
verify both conjectures. Many of these knots have toroidal
slopes as the largest and smallest non-trivial exceptional
slopes that we can use as our $b_1$ and $b_2$.
For the rest we have substantially complete knowledge 
of the boundary slopes for one of two reasons: 
first, if the knot
is two bridge or Montesinos, the boundary slopes are 
determined by Hatcher and Thurston~\cite{HatcherThurston} and 
Hatcher and Oertel~\cite{HatcherOertel1989}; and second, if
the knot has at most seven tetrahedra, we can extract the
boundary slopes from Culler's~\cite{CullerApoly} calculation 
of the $A$-polynomial. 

For each knot, we list the census name, the 
list of exceptional and boundary slopes both in standard and
SnapPy coordinates, and the name of the knot, if we know it.
We include all toroidal boundary slopes, which are exceptional slopes, but otherwise only provide sufficient boundary slopes to verify our two conjectures.

We label each boundary slope with one or more certificates 
C, B, K, L, M, or T: 
C if we use Culler's $A$-polynomial calculation, 
B or K for Brasile or Kabaya boundary slopes as
determined by SnapPy~\cite{Brasile, SnapPy, Kabaya} , 
L for the longitude, which is the 
boundary of a Seifert surface, 
M for boundary slopes of Montesinos
or two bridge knots~\cite{DunfieldProgram}, 
and T for the toroidal slopes, which are
exceptional slopes identified in Dunfield's census.
As discussed in Section~\ref{SecCensus}, 
both Conjectures~\ref{Conj1} and \ref{Conj6}
hold for these knots as we can verify using
our substantially complete knowledge of the boundary slopes.

For a few examples in this and the final \verb".csv" file, 
we do not provide the SnapPy coordinates, instead repeating
the standard coordinates. These are the knots for which
SnapPy does not identify $S^3$ surgery with $\frac10$
and instead uses some other rational number, 
typically $\frac01$ or $\pm \frac11$, as the trivial slope.

This leaves 336 knots, tabulated in 
\verb"Remaining336.csv", 
which have triangulations of eight or nine tetrahedra 
and non-trivial, non-toroidal exceptional slopes. 
For many of these, we were not able to verify one or both
of our conjectures. For this reason, we have added 
three columns to the data, indicating whether we have
verified Conjecture~\ref{Conj1}, Conjecture~\ref{Conj6}, 
or at least one of the two conjectures. 

Some of these knots have only one exceptional surgery meaning we 
cannot determine the linear transformation from
SnapPy coordinates to standard coordinates. Even in cases where
we found SnapPy coordinates for boundary slopes of these
knots, we cannot report the corresponding standard 
coordinates for those boundary slopes.

Note that taking the mirror reflection
of a knot changes the sign of all exceptional and boundary slopes. 
We do not distinguish between a knot and 
its reflection in our tables of surgery slopes.

\end{document}